\documentclass[11pt]{article}
\usepackage[T1]{fontenc}
\usepackage[utf8]{inputenc}
\usepackage{geometry}
\usepackage{amsmath}
\usepackage{amsthm}
\usepackage{enumitem}
\usepackage{setspace}
\usepackage[unicode=true,
 bookmarks=false,
 breaklinks=false,pdfborder={0 0 1},backref=false,colorlinks=false]
 {hyperref}
\usepackage[dvipsnames]{xcolor} 
\usepackage[color=Purple!50!white, textwidth=22mm]{todonotes}
\usepackage{comment}
\allowdisplaybreaks

\makeatletter
\theoremstyle{plain}
\newtheorem{thm}{\protect\theoremname}%
\theoremstyle{plain}
\newtheorem{lem}[thm]{\protect\lemmaname}
\theoremstyle{plain}

\theoremstyle{plain}
\newtheorem{cor}[thm]{\protect\corollaryname}
\newtheorem{problem}[thm]{\protect\problemname}

\newtheorem{innercustomgeneric}{\customgenericname}
\providecommand{\customgenericname}{}
\newcommand{\newcustomtheorem}[2]{%
  \newenvironment{#1}[1]
  {%
   \renewcommand\customgenericname{#2}%
   \renewcommand\theinnercustomgeneric{##1}%
   \innercustomgeneric
  }
  {\endinnercustomgeneric}
}

\newcustomtheorem{customthm}{Theorem}
\newcustomtheorem{customlemma}{Lemma}

\usepackage{amssymb}

\makeatother

\providecommand{\claimname}{Claim}
\providecommand{\lemmaname}{Lemma}
\providecommand{\theoremname}{Theorem}
\providecommand{\corollaryname}{Corollary}
\providecommand{\problemname}{Problem}

\setenumerate[0]{label=(\roman*), ref=\roman*}

\begin{document}
\title{A note on arithmetic progressions with restricted differences}
\author{David Conlon\thanks{Department of Mathematics, California Institute of Technology, Pasadena, CA 91125. Email:  {\tt dconlon@caltech.edu.} Research supported by NSF Award DMS-2348859.} \and Jacob Fox\thanks{Department of Mathematics, Stanford University, Stanford, CA 94305. Email: {\tt jacobfox@stanford.edu}. Research supported by NSF awards DMS-2154129 and DMS-2452737.} \and Huy Tuan Pham\thanks{Department of Mathematics, California Institute of Technology, Pasadena, CA 91125. Email:   {\tt htpham@caltech.edu.} Research supported by a Clay Research Fellowship.}}
\date{}
\maketitle

\begin{abstract}
In this note, we show how to adapt Tao's slice rank method to extend the Ellenberg--Gijswijt theorem on cap sets to the problem of forbidding arithmetic progressions with restricted differences. In particular, we show that if $q$ is an odd prime power, there is $\varepsilon_q>0$ such that if $S \subseteq \mathbb{F}_q$ with $0 \in S$ and $|S|>(q+1)/2$ and $A \subseteq \mathbb{F}_q^n$ contains no three-term arithmetic progression whose common difference is in $S^n$, then  $|A| \leq q^{(1-\varepsilon_q)n}$.
\end{abstract}

A celebrated theorem due to Roth~\cite{Roth} states that every subset of $[N]:=\{1,2,\ldots,N\}$ without a three-term arithmetic progression has size $o(N)$. The study of bounds for this problem has had a long history, culminating in a recent breakthrough of Kelley and Meka \cite{KM}, by whose work we now know that the largest such subset has size $N/e^{(\log N)^{\Theta(1)}}$. 

As shown by Meshulam \cite{Meshulam}, Roth's proof also adapts to finite abelian groups. In the finite field setting, where the problem, particularly in $\mathbb{F}_3^n$, is known as the cap set problem, the best known bound was dramatically improved by Ellenberg and Gijswijt \cite{EG} using a novel polynomial method originating in breakthrough work of Croot, Lev, and Pach \cite{CLP}. They proved that any subset of $\mathbb{F}_q^n$ without a three-term arithmetic progression has size at most $q^{(1-\varepsilon_q)n}$ for some $\varepsilon_q>0$. A beautiful way of formulating their argument was later found by Tao \cite{Tao} using a notion of rank for tensors that has since been termed the slice rank. 

A strengthening of Roth's theorem of considerable interest involves restricting the forbidden common difference to a special subset. In the finite field setting where we consider $\mathbb{F}_q^n$ with $q$ fixed, it is natural to consider the case where we fix a set $S \subseteq \mathbb{F}_q$ with $0 \in S$ and try to bound the size of the largest subset of $\mathbb{F}_q^n$ which forbids all three-term arithmetic progressions having a nonzero common difference from $S^n$. The case $S=\{0,1,2\}$ was studied by Bhangale, Khot, and Minzer \cite{BKM}. Improving on this together with Liu \cite{BKLM}, they addressed a question of Green \cite{Green} by giving a reasonable upper bound in the case $S=\{0,1\}$. Concretely, improving a weak bound coming from the density Hales--Jewett theorem~\cite{FK, Polymath}, they proved that $O_q((\log \log \log n)^{-c_q})$ is an upper bound for the density of the largest subset of $\mathbb{F}_q^n$ without an arithmetic progression with nonzero common difference in $\{0,1\}^n$.   

Green \cite{Green} even suggests that there may be a power-saving bound like that in the Ellenberg--Gijswijt theorem when $S = \{0,1\}$. The purpose of this brief note is to show how Tao's slice rank method can be used to prove a version of Roth's theorem with restricted differences in the finite field setting with such a power-saving bound, but under a stronger assumption on $|S|$ than that suggested by Green. 

\begin{thm}\label{threeapwithrestricteddifferences} 
For each odd prime power $q$, there is $\varepsilon_q>0$ such that if $S \subseteq \mathbb{F}_q$ with $0 \in S$ and $|S| >  (q+1)/2$ and $A\subseteq \mathbb{F}_q^n$ satisfies $|A| > q^{(1-\varepsilon_q)n}$, then $A$ contains a nontrivial three-term arithmetic progression whose common difference is in $S^n$.
\end{thm}

Blasiak, Church, Cohn, Grochow, Naslund, Sawin, and Umans \cite{BCCHGNSU} and, independently, Alon noticed that the argument of Ellenberg and Gijswijt extends to give an upper bound on the size of $d$-colored sum-free sets in $\mathbb{F}_q^n$ when $d = 3$. We begin by extending this result to higher $d$. Here we denote by $M(n,D,a)$ the number of solutions to $x_1+\cdots+x_n \leq D$ where each $0 \leq x_i \leq a$ is an integer. 

\begin{thm} \label{main2} If $\{(a_{1j},a_{2j},\ldots,a_{dj}) : 1 \le j \le N\}$ is a family of $d$-tuples of elements of $\mathbb{F}_q^n$ for which the only solutions to $a_{1j_1}+a_{2j_2}+\cdots+a_{dj_d}=0$ are when $j_1=\cdots=j_d$, then $N \leq dM(n,(q-1)n/d,q-1)$. 
\end{thm}

Note that the Ellenberg--Gijswijt theorem with $A=\{a_j\}_{j=1}^N$ a subset of $\mathbb{F}_q^n$ without a nontrivial three-term arithmetic progression is precisely the case of Theorem \ref{main2} with $d=3$ where we take $a_{1j}=a_{3j}=a_j$ and $a_{2j}=-2a_j$ for each $j \in [N]$.

For a field $\mathbb{F}$, a {\it $d$-dimensional tensor} over $\mathbb{F}$ is a map $T:A^d \rightarrow \mathbb{F}$. It is of {\it slice rank one} if it can be written in the form $T(x_1,\ldots,x_d)=f(x_i)g(x_1,\ldots,x_{i-1},x_{i+1},\ldots,x_d)$ for some $i \in [d]$. The {\it slice rank} of a $d$-dimensional tensor $T$ is then the minimum number of slice rank one maps needed to express $T$ as a linear combination.  For $d=2$, this can be shown to equal the usual rank of a matrix. A $d$-dimensional tensor is {\it diagonal} if $T(x_1,\ldots,x_d)$ is zero unless $x_1=x_2=\cdots=x_d$. An important result due to Tao~\cite{Tao}, now usually referred to as Tao's lemma, says that the slice rank of a diagonal tensor is equal to the number of nonzero entries. 

The following lemma, proved using the slice rank method, is the key to establishing our results. 

\begin{lem}\label{key}
Suppose that, for each $i \in [d]$, $A_i=(a_{ij})_{j=1}^N$ is a sequence of $N$ elements of $\mathbb{F}_q^n$. If $p$ is a polynomial of degree $D$ in $dn$ variables $x_{i\ell}$ with $(i,\ell) \in [d] \times [n]$ such that $p(a_{1j_1},a_{2j_2},\ldots,a_{dj_d}) \neq 0$ if and only if $j_1=j_2=\cdots=j_d$, then $N \leq dM(n,D/d,q-1)$.  
\end{lem}

\begin{proof}
Let $P$ be the polynomial in $dn$ variables which is formed from $p$ by replacing the exponent of each variable that appears with exponent $m \geq q$ with the remainder when $m$ is divided by $q-1$, except when $(q-1)\mid m$ where we replace the exponent with $q-1$. As $x^q=x$ in $\mathbb{F}_q$, $P$ and $p$ have the same values on all inputs. Then $P$ is a polynomial of degree at most $D$ with degree at most $q-1$ in each variable. 

Consider the tensor $T:[N]^d \to \mathbb{F}_q$ with $T(j_1,j_2,\ldots,j_d)=P(a_{1j_1},a_{2j_2},\ldots,a_{dj_d})$. By assumption, $T$ is a diagonal tensor with all $N$ diagonal entries nonzero, so, by Tao's lemma, the slice rank of $T$ is $N$. On the other hand, each monomial of $P$ is the product of $dn$ variables, each variable with degree at most $q-1$, and total degree at most $D$. Hence, for each monomial of $P$, there is some $i \in [d]$ such that the total degree coming from the $n$  variables $x_{i\ell}$ with $\ell \in [n]$ is at most $D/d$. From here, we argue as in \cite{Tao} to conclude that the slice rank of $P$ is at most $dM(n,D/d,q-1)$, which completes the proof.
\end{proof}

Let $f({\bf x_1},{\bf x_2},\ldots,{\bf x_d}):=\prod_{\ell=1}^n \left(1-(x_{1\ell}+x_{2\ell}+\cdots+x_{d\ell})^{q-1}\right)$, where each ${\bf x_i}=(x_{i1},\ldots,x_{in}) \in \mathbb{F}_q^n$. Letting $p=f$ in Lemma~\ref{key}, we get the multicolor sum-free theorem, Theorem \ref{main2}.  

The main additional observation needed for the proof of Theorem~\ref{threeapwithrestricteddifferences} is that if we let $p=fg$, where $g$ is a polynomial of not too large degree, we can still obtain a good upper bound on the size of our sets, which can be seen as being multicolor sum-free with additional constraints. To estimate how much room we have, we make use of the following lemma, due to L.~M.~Lov\'asz and Sauermann \cite{LS}, which gives a good upper bound on $M(n,\alpha(m-1)n,m-1)$. 

\begin{lem}\label{lslemma}
For every $0 < \alpha < 1/2$ and positive integer $m$,  
$$|\{(b_1,\dots,b_n) \in \{0,1,\ldots,m-1\}^n : b_1+\cdots+b_n \leq \alpha (m-1)n \}| \leq (\Gamma_{\alpha,m})^n,$$
where 
$$\Gamma_{\alpha,m}=\inf_{0 < \gamma < 1} \gamma^{-\alpha (m-1)}\left(1+\gamma+\cdots+\gamma^{m-1}\right).$$
\end{lem}

As pointed out in \cite{LS}, we have $\Gamma_{\alpha,m}<m$ as the function above has value $m$ and positive derivative at $\gamma=1$, where, for the latter property, we used that $\alpha<1/2$. We will apply this with $m=q$ a prime power. 

\begin{lem} \label{main3} 
Suppose $0 < \alpha<1/2$, $\{(a_{1j},a_{2j},\ldots,a_{dj}) : 1 \le j \le N\}$ is a family of $d$-tuples of elements of $\mathbb{F}_q^n$, and $g$ is a polynomial of degree at most $(d\alpha-1)(q-1)n$. If  $a_{1j_1}+a_{2j_2}+\cdots+a_{dj_d}=0$ and $g(a_{1{j_1}},a_{2{j_2}},\ldots,a_{d{j_d}}) \neq 0$ hold if and only if $j_1=\cdots=j_d$, then $N \leq d(\Gamma_{\alpha,q})^n$. 
\end{lem}

Lemma \ref{main3} follows from Lemmas \ref{key} and \ref{lslemma}, by noting that the polynomial $p=fg$ has degree $D$ at most $(q-1)n+(d\alpha-1)(q-1)n =d\alpha(q-1)n$, so that $D/d \leq \alpha(q-1)n$.  

We can improve the bound in Lemma \ref{main3} slightly, removing the factor $d$ using a tensor product trick. 

\begin{thm} \label{main2'} With the same assumptions as in Lemma~\ref{main3}, $N \leq (\Gamma_{\alpha,q})^n$.

\end{thm}
\begin{proof}
Assume for contradiction that, for some $n>0$, there exists a sequence of $d$-tuples $A \subseteq (\mathbb{F}_q^n)^d$ of size $N > (\Gamma_{\alpha,q})^n$ such that the only solutions to $a_{1j_1}+\dots+a_{dj_d}=0$ and $g(a_{1j_1},\dots,a_{dj_d})\ne 0$ are when $j_1=\dots=j_d$. For a positive integer $k$, let $n'=kn$ and consider the sequence of $d$-tuples $\overline{A}=A^k \subset (\mathbb{F}_q^{nk})^d$ with $|\overline{A}| = N^{k}$. Each tuple $\overline{a}$ of $\overline{A}$ can be seen as a concatenation $\overline{a} = [\overline{a}^1,\dots,\overline{a}^k]$ of $k$ tuples of elements $\overline{a}^h = (a_{1j},\dots,a_{dj}) \in (\mathbb{F}_q^n)^d$ in $A$. Let $\overline{g}$ be the polynomial over $(\mathbb{F}_q^{n'})^d$ given by $\overline{g}(\overline{x}_1,\dots,\overline{x}_d) = \prod_{h=1}^{k} g(\overline{x}_1^h,\dots,\overline{x}_d^h)$, where we write $\overline{x} = [\overline{x}^1,\dots,\overline{x}^{k}]$ for $\overline{x}\in (\mathbb{F}_q^{n'})^d$. The degree of $\overline{g}$ is $k\deg(g) \le k(d\alpha - 1)(q-1)n = (d\alpha -1)(q-1)n'$.

We have $\overline{g}(\overline{a}_1,\dots,\overline{a}_{d}) \ne 0$ if and only if $g(\overline{a}_{1}^h,\dots,\overline{a}_d^h) \ne 0$ for all $h \in [k]$. Furthermore, $\overline{a}_1+\dots+\overline{a}_{d} = 0$ if and only if $\overline{a}_1^h+\dots+\overline{a}_{d}^h=0$ for all $h \in [k]$. By our assumption, for both of these conditions to hold, we must have that, for all $h\in [k]$, $(\overline{a}_1^h,\dots,\overline{a}_d^h) = (a_{1j_h},\dots,a_{dj_h})$ for some $j_h$. In particular, $\overline{A}$ satisfies the conditions of Lemma \ref{main3} with respect to the polynomial $\overline{g}$. However, for $k$ sufficiently large, $|\overline{A}| = N^k > d(\Gamma_{\alpha,q})^{nk} = d(\Gamma_{\alpha,q})^{n'}$, which contradicts Lemma \ref{main3}. 
\end{proof}

The following is a corollary of Theorem \ref{main2'} and implies Theorem \ref{threeapwithrestricteddifferences} by taking $S_1=\cdots=S_n=S$. 

\begin{cor}
Let $q$ be an odd prime power, and suppose $\{0\} \subseteq S_1,\ldots,S_n \subseteq \mathbb{F}_q$ and $A=\{a_j\}_{j=1}^N \subseteq \mathbb{F}_q^n$ is a set which contains no nontrivial three-term arithmetic progression with common difference in $S_1 \times \cdots \times S_n$. Let $\mu=\frac{1}{n}\sum_{j=1}^n (q-|S_j|)$ be the average size of the complements $\bar{S_j}$ and $\alpha = 1/3+\mu/(3(q-1))$. If $\alpha < 1/2$, then $|A| \leq (\Gamma_{\alpha,q})^n$.   
\end{cor}
\begin{proof}
We will apply Theorem \ref{main2'}. Let $d=3$, and, for $j \in [N]$, let $a_{1j}=a_{3j}=a_j$ and $a_{2j}=-2a_j$. 
Let  $g({\bf x_1},{\bf x_2},{\bf x_3})=\prod_{\ell=1}^n \prod_{t \in \mathbb{F}_q \setminus S_\ell} ((x_{3\ell}-x_{1\ell})/2-t)$, where $x_{i\ell}$ is the $\ell^{\textrm{th}}$ coordinate of ${\bf x_i}$. Note that $g$ has degree $\mu n = (3\alpha-1)(q-1)n$. As there are no nontrivial three-term arithmetic progressions in $A$ with common difference in $S_1 \times \cdots \times S_n$, we have that $a_{1j_1}+a_{2j_2}+a_{3j_3}=0$ and $g(a_{1j_1},a_{2j_2},a_{3j_3}) \ne 0$ if and only if $j_1=j_2=j_3$. Indeed, $a_{1j_1}+a_{2j_2}+a_{3j_3}=0$ is equivalent to $a_{j_1},a_{j_2},a_{j_3}$ forming a three-term arithmetic progression. If $a_{j_1},a_{j_2},a_{j_3}$ form a three-term arithmetic progression, then the common difference is $(a_{j_3}-a_{j_1})/2=(a_{3j_3}-a_{1j_1})/2$ and the common difference is in $S_1 \times \cdots S_n$ if and only if $g(a_{1j_1},a_{2j_2},a_{3j_3}) \neq 0$. By Theorem \ref{main2'}, if $\alpha < 1/2$, then $|A|\leq (\Gamma_{\alpha,q})^n$. 
\end{proof}

We conclude with two open problems. The first asks if we can improve the range of $|S|$ in Theorem~\ref{threeapwithrestricteddifferences}, a problem which seems to already be interesting for $|S| = \lfloor (q+1)/2 \rfloor$, but, as suggested by Green~\cite{Green}, may be true all the way down to $|S| = 2$.

\begin{problem}
Improve the range of $|S|$ in Theorem \ref{threeapwithrestricteddifferences}. 
\end{problem}

Our second problem stems from the fact that the multicolor sum-free result was used~\cite{FL, FLS} to prove an arithmetic triangle removal lemma in $\mathbb{F}_q^n$. In line with this, it is natural to wonder if there is a restricted differences version extending Theorem \ref{main2}. 

\begin{problem}
Is there an arithmetic triangle removal lemma with restricted differences? 
\end{problem}

\end{document}